\documentclass{amsart}

\usepackage{amsmath, amssymb, amsthm}
\usepackage{mathrsfs}
\usepackage{graphicx}
\usepackage{graphics}
\usepackage{amsfonts}
\usepackage{color}
\usepackage{hyperref}

\newtheorem{theorem}{Theorem}
\newtheorem{lemma}[theorem]{Lemma}

\theoremstyle{definition}

\newtheorem{question}[theorem]{Question}
\newtheorem{claim}[theorem]{Claim}
\newtheorem{conjecture}[theorem]{Conjecture}

\theoremstyle{remark}
\newtheorem{remark}[theorem]{Remark}

\numberwithin{equation}{section}

\newcommand{\cocf}{$co\mathcal{CF}\,$}

\newcommand{\Z}{\mathbb{Z}}

\newcommand{\LF}{\mathcal{L}}

\newcommand{\quatT}{{\mathrm{QAut}}(\mathcal{T}_{2,c})}

\newcommand{\cT}{\mathcal{T}_{2,c}}
\newcommand{\supp}{\mathrm{Supp}}

\newcommand{\words}{\left\{0,1\right\}^\ast}
\newcommand{\CS}{\mathfrak{C}}

\begin{document}

\title{Embeddings into Thompson's group $V$ and \cocf groups}

\author{Collin Bleak}
\address{School of Mathematics and Statistics, University of St Andrews,
Mathematical Institute, North Haugh, St Andrews, KY16 9SS, Scotland, UK}
\email{collin@mcs.st-and.ac.uk}

\author{Francesco Matucci}
\address{D\'epartement de Math\'ematiques,
B\^atiment 425, bureau 21,
Facult\'e des Sciences d'Orsay, Universit\'e Paris-Sud 11,
F-91405 Orsay, France}
\email{francesco.matucci@math.u-psud.fr}

\author{Max Neunh\"offer}
\address{School of Mathematics and Statistics, University of St Andrews,
Mathematical Institute, North Haugh, St Andrews, KY16 9SS, Scotland, UK}
\email{neunhoef@mcs.st-and.ac.uk}

\subjclass[2000]{Primary XYZ; Secondary XYZ}



\begin{abstract}
Lehnert and Schweitzer show in \cite{lehnschweitz1}
that R. Thompson's group $V$ is a co-context-free 
(\cocf) group, thus implying that all of its finitely generated
subgroups are also \cocf groups.
Also, Lehnert shows in his thesis that $V$ embeds inside 
the \cocf group $\mathrm{QAut}(\mathcal{T}_{2,c})$, which is a group of particular bijections on the vertices of an infinite binary $2$-edge-colored tree, and he conjectures that $\mathrm{QAut}(\mathcal{T}_{2,c})$
is a universal \cocf group. We show that $\mathrm{QAut}(\mathcal{T}_{2,c})$ embeds into $V$, and thus obtain a new form for Lehnert's conjecture.
Following up on these ideas, we begin work to build a representation theory into R. Thompson's group $V$.  In particular we classify precisely which Baumslag-Solitar groups
embed into $V$.
\end{abstract}

\maketitle

\section{Introduction}
\subsection{History and context}
There has been a long historical interplay between classes of formal languages, and classes of groups.  This connection was first made by Max Dehn, who in 1911 stressed the importance of several formal problems associated with group presentations, one of which was the word problem for groups (Given a group $G$, is there an algorithm which determines, in finite time, whether or not any given finite product of generators is trivial?).  Note that when one looks at the set of ``words'' in the generators which are equivalent to the identity in the group one has specified a formal language.  The more complicated this formal language is, the more complex an algorithm would have to be in order to positively answer Dehn's question.

Thus, as classes of languages become more complex, some ``corresponding'' classes of groups become wider, and thus harder to comprehend in a meaningful fashion.  Ways to build these correspondences are through the word or co-word problems for groups, but other flavours of correspondence have also been seen (see, e.g., \cite{anisimov,muller-schupp-2, holtrerotho,holt-owens-thomas, horerocfconj,horoind}).  To date, classifications of corresponding sets of groups (for classes of languages) only exist for very simple classes of languages, but the results on the group theory side are quite striking.  In order to discuss this further, we need to give a definition.

Given a finitely generated group $G=\langle X \rangle$, one can define the \emph{language of the
word problem} to be the set of words
\[
WP(G)=\{w \in F(X) \mid w \equiv_G 1 \}.
\]
Similarly, the \emph{language of the co-word problem} is defined to be
\[
coWP(G)=\{w \in F(X) \mid w \not \equiv_G 1 \}.
\]
And now, let us state some results on the group theory side.

In 1972, Anisimov in \cite{anisimov} shows that for a finitely generated group $G$,
$WP(G)$ is a regular language if and only if $G$ is finite.  The proof of this is not difficult, but the idea of building the correspondence in the first place represents quite a step forward.  Later,  in the early 1980's, a celebrated collection of papers of Muller and Schupp (relying on Dunwoody's accessibility theory - see \cite{dunwoody1, muller-schupp-1, muller-schupp-2}) show the following theorem.

\begin{theorem} [Muller, Schupp]
Let $G$ be a finitely generated group.  Then, $WP(G)$ is a context-free language 
if and only if $G$ is virtually free.
\end{theorem}
A group $G$ with $WP(G)$ a context free language is called a \emph{context-free group} or a $\mathcal{CF}$ group.
 
One thus sees that $\mathcal{CF}$ groups provide a generalisation of finite groups from the point of view of
computer science, as the context-free languages are one of the simplest generalisations of regular languages (the difference arises as the machines that are used to create context-free languages are directed labelled graphs which have a stack for memory, and can make transitions based on this changing stack, while the machines for regular languages are simply finite directed graphs with no form of memory other than their structure, see, e.g., \cite{Hopcroft+Ullman/79/Introduction} for an introduction to automata and formal language theory).  

A further generalisation is given by the class \cocf of \emph{co-context-free groups}
which are defined to be finitely generated groups $G$ such that $coWP(G)$ is context-free.  This is a generalisation as the virtually free groups actually have \emph{deterministic} context-free word problems, and thus their language of co-words is also a deterministic context-free language.

Muller and Schupp's results are in some sense the last complete classification of a class of groups corresponding to a class of languages.  Thus, focus has shifted to the \cocf groups.

In \cite{holtrerotho} Holt, Rees, R\"over and Thomas  
introduce the \cocf groups.  They show that the class
\cocf of all co-context-free groups is closed under taking:
\begin{itemize}
\item taking finite direct products, 
\item taking restricted standard wreath products with context-free top groups, 
\item passing to finitely generated subgroups 
\item passing to finite index overgroups.
\end{itemize}
In \cite{holtrerotho} there are also various conjectures about how other operations interact with the class \cocf.  Also, there is a discussion about whether certain very specific groups can be \cocf groups.  Currently, it is conjectured that $\Z^2*\Z$ and the Grigorchuk group $\Gamma$ are not in \cocf, and that certain wreath products cannot be in \cocf.

Now, let 
$\mathcal{T}_{2,c}$ be the infinite binary $2$-edge-colored binary tree (left edges $=$ red, right edges $=$ blue), and let $\quatT$ be the group of all bijections on the vertices of $\mathcal{T}_{2,c}$
which respect the edge and color relationships, except for at possibly finitely many locations.
Lehnert in \cite{LehnertDissertation} shows the following results, amongst others.

\begin{theorem}[Lehnert] The group $\quatT$ is a \cocf group, and 
there is an embedding from R. Thompsons group $V$ into $\quatT$.
\end{theorem}

We note that Lehnert and Schweitzer \cite{lehnschweitz1} prove that the Higman-Thompson groups $G_{r,s}$ are in \cocf, which also shows that $V= G_{2,1}$ is in \cocf.

In his dissertation, Lehnert also makes the following conjecture:

\begin{conjecture}[Lehnert]
The group $\quatT$ is a universal \cocf group.
\end{conjecture}
Thus, Lehnert conjectures that a group $G$ is in \cocf if and only if it is finitely generated and it embeds in $\quatT$.

It is the main focus of this paper to discuss Lehnert's conjecture.
 
\subsection{Our results, and the ongoing discussion}
Lehnert and Schweitzer in \cite{LSPersonalCommunication} asked the authors of \cite{bleak-salazar1}  whether they thought one could embed $\quatT$ into R. Thompson's group $V$, as those authors had just shown that $\Z^2*\Z$ fails to embed into $V$ (supporting the conjecture of Holt, R\"over, Rees and Thomas).  This question eventually lead the current authors to our main result, below.
\begin{theorem}
There is an embedding $\quatT \rightarrowtail V$.
\end{theorem}
That is, one can now re-state Lehnert's conjecture as:

\begin{conjecture}[Lehnert]
R. Thompson's group $V$ is a universal \cocf group.
\end{conjecture}

Thus, if Lehnert's conjecture is true, a group will be in \cocf if and only if it is finitely generated and it embeds as a subgroup of $V$.

Working to understand the class \cocf better, we will now discuss some of what is known about the subgroups of R. Thompson's group $V$. 

 It is known that $V$ contains many embedded copies of non-abelian free groups, and indeed, many free products of its subgroups.  In the paper \cite{bleak-salazar1}, the authors give some more specific results.  They find conditions under which particular restricted wreath products and free products of subgroups of $V$ actually can embed into $V$. Also, they show that $\mathbb{Z} \ast \mathbb{Z}^2$
does not embed into $V$, supporting the conjecture of Holt, R\"over, Rees, and Thomas.  In \cite{corwinthesis} Nathan Corwin adds to these results by showing that $\mathbb{Z} \wr \mathbb{Z}^2$ does not embed
into $V$.

Three important questions that are often asked, pertaining to subgroup structure of a given group $G$, are a) to decide if non-abelian free groups embed into $G$, b) to decide whether surface groups embed into $G$, and c) to decide if the Baumslag-Solitar groups embed into $G$.  

In \cite{roverthesis} R\"over shows that if $n$ is a proper divisor of $m$, then the Baumslag-Solitar group $BS(m,n)$
does not embed into $V$.  R\"over's proof is based on using the fact observed by Higman (the corollary to Lemma 9.3 (see \cite{hig})) that non-torsion elements in the Higman-Thompson groups fail to have infinitely many roots.
Using a different method, we extend R\"over's result as follows.

In this paper we show that certain Baumslag-Solitar groups are co$\mathcal{CF}$ groups.  In particular,  we decide exactly which Baumslag-Solitar groups embed in R. Thompson's group $V$ (and indeed, into the Higman-Thompson groups $G_{s,r}$).

\begin{theorem}\label{BS-groups}
Let $m,n \in \Z \setminus \{ 0 \}$. Let $BS(m,n)$ be the corresponding
Baumslag-Solitar group.
\begin{enumerate}
\item If $|m| \ne |n|$, then $BS(m,n)$ fails to embed in $V$.
\item If $|m| = |n|$, then there is an embedding of $BS(m,n)$ in $V$.
\end{enumerate}
\end{theorem}

In fact, the theorem above holds for all of the Higman-Thompson groups $G_{r,s}$ (not just $V=G_{2,1}$) using essentially the same proof we give.

At the time of this writing, Burillo, Cleary, R\"over, and Stein are working on a survey of obstructions to finding embeddings into R. Thompson's group $V$, which will feature various arguments including an obstruction based on distortion of subgroups.
One can use the distortion obstruction to give the above non-embedding results for the Baumslag-Solitar groups, and in fact the argument is equivalent to our presented argument.  We discuss this briefly in Section \ref{BSEmbeddings}.  

It is a question of Gromov as to whether the surface groups embed into all (word-)hyperbolic groups \cite{bestvina-questions}, and there has been much work on this question by the broader community.  Amongst many results and partial results, one can highlight the general results of the Calegari school on stable commutator length, and in particular the work in \cite{calegari-walker-2, calegari-walker-1, walkerthesis} which shows amongst other things that in HNN extensions of free groups, amalgamating the base free group to an endomorphic embedded copy of itself, one can often find surface subgroups.  The present authors have attempted to use these results to find surface groups (other than the torus and the Klein bottle groups)
in $V$, but so far we have been unsuccessful.  Thus, we ask the following question.
\begin{question}
Do the hyperbolic (closed) surface groups embed into R. Thompson's group $V$?
\end{question}
We note in passing that it is known that the set of subgroups in $V$ is closed in passing to finite index over-groups (see e.g. \cite{bleak-salazar1,roverthesis}), thus all that is required is to find a copy of any single hyperbolic (closed) surface subgroup  in $V$ to see that they all embed.

\subsection*{Acknowledgements} Thanks go to Lehnert and Schweitzer for asking
us about the relation between $V$ and $\quatT$. We wish to thank
Fr\'ed\'eric Haglund for suggesting us to investigate whether or not surface groups embed in $V$, and to thank Mark Sapir for asking us to investigate whether the Baumslag-Solitar groups embed into $V$.
We thank Jos\'e Burillo, Yves de Cornulier and Claas R\"over for helpful and interesting conversations.
The second author gratefully acknowledges the Fondation Math\'ematique 
Jacques Hadamard (FMJH - ANR - Investissement d'Avenir) for 
the support received during the development of this work.

\section{Decomposition of elements of $\mathrm{QAut}(\mathcal{T}_{2,c})$}
\label{sec:decomposition-QAut}

In this section we define two ways to present an element of $\quatT$.  The first form associates a minimal element of $V$ with an element of $\quatT$, but is complicated by the intervention of two further bijections between finite subsets of $\{0,1\}^*$, while the second way is simpler, and associates any element of $\quatT$ to a non-unique element of $V$ and a bijection between two finite subsets of $\{0,1\}^*$.  We believe our first form is new while the second appears to be what is used by Lehnert in his dissertation \cite{LehnertDissertation}.  Still, the first form enables us to build an embedding of $\quatT$ into $V$ in
Section \ref{sec:QAut-inside-V}.

To set up the notation we will be using, we define by $\{0,1\}^\ast$ as the set of all finite words in the alphabet $\{0,1\}$ and by $\{0,1\}^{\omega}$
the set of infinite words in the alphabet $\{0,1\}$.
The set $\{0,1\}^{\omega}$ corresponds to the boundary of the tree $\mathcal{T}_2$ and to the standard ternary Cantor set $\mathfrak{C}$.
If a word $a \in \{0,1\}^\ast$ is a \emph{prefix} of a word $w$, we write $a<w$.
If $a_1\neq a_2\in \words$ we write $a_1 \perp a_2$ if neither is a prefix of the other.

\subsection{Building an element in $V$ from an element in $\quatT$}

Let $\tau \in \mathrm{QAut}(\mathcal{T}_{2,c})$ be seen as a map $\tau:\{0,1\}^\ast 
\twoheadrightarrow \{0,1\}^\ast$. 
For any $w \in \{0,1\}^\ast$ we examine the pair $(w, w \tau)$ to find the largest common suffix $s_w$ such that
\begin{align*}
w = x_w s_w \\
w \tau = y_w s_w
\end{align*}
for suitable $x_w, y_w \in \{0,1\}^\ast$ prefixes. We define $\Gamma_\tau := \{(x_w,y_w) \mid w \in \{0,1\}^\ast \}$.

\claim{$\Gamma_\tau$ is finite.}

\begin{proof}
Since the map $\tau$ is a quasi-automorphism, there exists a level $k$ in the domain tree below which the adjacency and color relations are respected by the action of $\tau$. 
Assume that $w_1$ is any node below level $k$ and let $w_2:=w_1 a$, for some $a \in \{0,1\}$. Since $w_1$ and $w_2$
are adjacent and below level $k$, we have $w_2 \tau =\left(w_1 a \right) \tau =
\left( w_1 \tau \right) a$. 
This immediately extends to any word $\lambda \in \words$, so that
$\left(w_1 \lambda \right) \tau = \left( w_1 \tau \right) \lambda$. Since this argument holds for any $w_1$ below level $k$, this shows that there
can be only finitely many elements in $\Gamma_\tau$ since every word below level $k$ is a descendant of the finite set of words at level $k$.
\end{proof}

Since the set $\Gamma_\tau$ is finite we can find a subset $M_\tau \subseteq \Gamma_\tau$ which is essential in the following sense: for every pair
$(a,b)\in M_\tau$ 
there exist infinitely many words $w$ such that 
$(x_w,y_w)=(a,b)$. 

Recall that an \emph{anti-chain} is a subset of a partially ordered set such that any two elements in the subset are incomparable
and that an anti-chain is \emph{complete} if it is maximal with respect to inclusion.
By the definition of $M_\tau$ the sets of words
\[\mathcal{L}_{D_\tau}:=\left\{a\in\words\mid \exists b\in \words, (a,b)\in M_\tau\right\}
\]
and
\[\mathcal{L}_{R_\tau}:=\left\{b\in\words\mid \exists a\in \words, (a,b)\in M_\tau\right\}
\]
both form finite complete anti-chains for the poset $\words$ ordered by prefix inclusion. The set 
$\mathcal{L}_{D_\tau}$ has the properties that if $a_1\neq a_2\in 
\mathcal{L}_{D_\tau} \subset\words$ then  
$a_1\perp a_2$ and for all sufficiently long words $w$ in $\words$, there is $a_w\in 
\mathcal{L}_{D_\tau}$ so that $a_w<w$, and similarly for $\mathcal{L}_{R_\tau}$.
Therefore, $M_\tau$ naturally determines a prefix code map on $\{0,1\}^{\omega}$ which is determined by two finite complete anti-chains of equal cardinality.  Note that this is another way of defining element of R. Thompson's group $V$.

Furthermore, one easily sees that the construction of $M_\tau$ as above produces the bijection between the leaves of the unique minimal tree-pair representative for the particular element $v_\tau$ of $V$ which $M_\tau$ determines.  As any finite rooted subtree of $\mathcal{T}_2$ is determined by its leaves and vice-versa, 
we can identify the finite complete anti-chains $\mathcal{L}_{D_\tau}, \mathcal{L}_{R_\tau}$ with 
finite trees $D_\tau, R_\tau$ and say that  $(D_\tau,R_\tau,\sigma_\tau)$ is the minimal tree pair representative for the element $v_\tau$ that $M_\tau$ determines.  We write $v_\tau\sim (D_\tau,R_\tau,\sigma_\tau)$ in this case to emphasise that $(D_\tau,R_\tau,\sigma_\tau)$ is the minimal tree pair representing $v_\tau$. (Note: we also write $v_\tau\sim(D,R,\sigma)$ if $(D,R,\sigma)$ is any tree pair representing $v_\tau$.)

\begin{remark}
There may exist a word $w \in \{0,1\}^\ast$ such that there is a pair $(a,b) \in M_\tau$ with $a$ prefix of $w$ but 
$w \to w \tau$ is not determined by the pair $(a,b)$.  This will lead us to consider a finite permutation also associated with $\tau$, which we will discuss below in Subsection \ref{ssec:permutation-part-of-tau}.
\end{remark}

\subsection{The finite bijection between internal nodes of the tree pair in $v_\tau$.
\label{ssec:bijection-part-of-tau}}


By a slight abuse of notation we identify the tree $D_\tau$ with the set of words corresponding to the nodes of $\tau$ strictly above the leaves. We make a similar identification for $R_\tau$.

We now define a (non-canonical) bijection $b_\tau: D_\tau \to R_\tau$. For every word $w \in D_\tau \cap (R_\tau)\tau^{-1}$, we define $w b_\tau:= w \tau$.
Then we complete $b_\tau$ to a bijection by choosing and fixing a bijection between the sets $D_\tau \setminus (R_\tau) \tau^{-1}$ and $R_\tau \setminus (D_\tau)\tau$.

\subsection{Finite permutation on a subset of the nodes of the tree $R_\tau$\label{ssec:permutation-part-of-tau}}
We observe that $v_\tau$ can be seen as a tree pair diagram or a map acting on $\words \setminus D_\tau$
as a prefix replacement map.
Consider the following map 
\[
\widetilde{v}_\tau
=
\begin{cases}
b_\tau & \mbox{over $D_\tau$} \\
v_\tau & \mbox{over $\{0,1\}^\ast \setminus D_\tau$.}
\end{cases}
\]
The map $\widetilde{v}_\tau$ is an element of $\mathrm{QAut}(\mathcal{T}_{2,c})$ and, by construction, it differs with $\tau$ on only
finitely many vertices (which are possibly spread between $D_\tau$ and $\{0,1\}^\ast \setminus D_\tau$). Thus the map
$p_\tau := \widetilde{v}_\tau^{-1} \tau$ is a permutation on finitely many vertices of $\cT$.

\subsection{Minimal decomposition of elements in $\mathrm{QAut}(\mathcal{T}_{2,c})$}

We are now able to write down the decomposition for $\tau \in \mathrm{QAut}(\mathcal{T}_{2,c})$ that we were looking for. The following result is an immediate consequence
of the discussion above.

\begin{lemma}[Minimal decomposition]
\label{thm:minimal-decomposition}
For every $\tau \in \mathrm{QAut}(\mathcal{T}_{2,c})$, there exists a permutation $p_\tau$ on finitely many vertices of $\mathcal{T}_{2,c}$ so that
\begin{equation}\label{eq:minimal-decomposition}
\tau = \widetilde{v}_\tau p_\tau
\end{equation}
where $\widetilde{v}_\tau \in \mathrm{QAut}(\mathcal{T}_{2,c})$ acts as an element of Thompson's group $V$ beneath a suitable level and is a bijection on
the finitely many nodes above such level.
\end{lemma}

We say that the decomposition of Lemma \ref{thm:minimal-decomposition} is a \emph{minimal decomposition} because the tree pair for the associated element of $V$ is minimal.

\begin{remark}
As observed above, this decomposition is not unique and depends on how we choose to build the map $b_\tau$. Nevertheless, there is always a way to create
the decomposition in (\ref{eq:minimal-decomposition}).
\end{remark}

\subsection{Disjoint decomposition form for elements in $\mathrm{QAut}(\mathcal{T}_{2,c})$}
It is possible to rewrite the form of Lemma \ref{thm:minimal-decomposition} so that $v_\tau$ is represented by a tree pair $(D_d,R_d,\sigma_d)$ whose domain tree $D_d$ is the full subtree of $\mathcal{T}_2$ at depth $k$, for some $k$, while $b_\tau:X\to Y$ is a bijection from the set $X$ of vertices of  $\mathcal{T}_2$ of depth less than $k$, to the set $Y$ of vertices of $\mathcal{T}_2$ which are above the leaves of $R_d$,  and with $p_\tau$ equal to the identity map.

\begin{lemma}[Disjoint decomposition]
\label{thm:disjoint-decomposition}
For every $\tau \in \mathrm{QAut}(\mathcal{T}_{2,c})$, there exists 
a map $d_\tau \in \mathrm{QAut}(\mathcal{T}_{2,c})$ which 
acts as an element of Thompson's group $V$ beneath a suitable level and is a bijection on
the finitely many nodes above such level and
such that
\begin{equation}\label{eq:disjoint-decomposition}
\tau = d_\tau.
\end{equation}
\end{lemma}

\begin{proof}
We use the same notation of the previous subsections and apply Lemma \ref{thm:minimal-decomposition} 
to the element $\tau$ to rewrite it as $\widetilde{v}_\tau p_\tau$. We then unreduce the tree pair diagram $(D_\tau,R_\tau,\sigma)$ of 
$v_\tau$ to a new pair $(D_\tau',R_\tau',\sigma')$ where 
$D_\tau'$ is a full subtree chosen so that
the leaves of each of the trees $D_\tau'$ and $R_\tau'$ are strictly below 
the set of vertices non-trivially acted upon by both 
$b_\tau$ and $p_\tau$.
We now define a map $b_\tau': D_\tau' \to R_\tau'$ by $(t)b_\tau':=(t)\widetilde{v}_\tau$
and the map $p_\tau': R_\tau' \to R_\tau'$ by $(t)p_\tau':=(t)p_\tau$.
By construction and Lemma \ref{thm:minimal-decomposition}, it is obvious that $(t)\tau=(t)b_\tau' p_\tau'$ for $t \in D_\tau'$.

We thus define $s_\tau:D_\tau' \to R_\tau'$ by $(t)s_\tau:=b_\tau' p_\tau'$ and finally we define
\[
d_\tau
=
\begin{cases}
s_\tau & \mbox{over $D_\tau'$} \\
\widetilde{v}_\tau & \mbox{over $\{0,1\}^\ast \setminus D_\tau'$.}
\end{cases}
\]
It is immediate from our construction that $\tau = d_\tau$ and that $d_\tau$ is built as an element of $V$ below some level and a bijection
above such level.
\end{proof}

We say that the decomposition of Lemma \ref{thm:disjoint-decomposition} is a \emph{disjoint decomposition} because 
$d_\tau$ is described via a bijection $s_\tau$ and an element of Thompson's group $V$, that is $\widetilde{v}_\tau$ restricted
to the lowest vertices in $\words \setminus D_\tau'$ (which clearly define a tree).
We sometimes refer to $s_\tau$ as the \textbf{bijection part of $d_\tau$} and to element of the Thompson's group $V$
given by restricting $\widetilde{v}_\tau$ to the lowest vertices of $\words \setminus D_\tau'$
as the \textbf{$V$-part of $d_\tau$}.

\begin{remark}
Given a map $\tau \in \quatT$ we can canonically define \textbf{cutoff level of $\tau$}
as the smallest level $k$ of the domain tree such that the map $\tau$ behaves as an automorphism on every vertex on every level $\ge k$. If $I(\tau)$ is the set of vertices in the domain of $\tau$
such that $\tau$ does not respect either the adjacency relation or the color relation,
and $\ell $ is the largest level of any point inside $I(\tau)$, then $k=\ell +1$.

We observe that the definition of the cutoff level depends only on $\tau$, but one can also recover it via the support of $p_\tau$.
If we denote by $\supp(p_\tau)$ the support of $p_\tau$ and consider the set
\[
Z(p_\tau):=\left(D_\tau \setminus (R_\tau) \tau^{-1} \right) \cup \left(R_\tau \setminus (D_\tau)\tau \right)\tau^{-1} \cup 
\supp(p_\tau) \tau^{-1}
\]
and $r$ is the largest level of any point in $Z(p_\tau)$, then $k=r+1$.
If $k$ is the cutoff level of $\tau$, we call the disjoint decomposition built via the full subtree of level $k$
as the \textbf{cutoff disjoint decomposition of $\tau$}. We observe that this decomposition is unique
since the cutoff level is uniquely defined.
\end{remark}

\begin{lemma}\label{thm:cutoff-disjoint}
Any disjoint decomposition can be obtained by refining the cutoff disjoint decomposition..
\end{lemma}

\begin{proof}
This is immediate from the proof of Lemma \ref{thm:disjoint-decomposition} and the definition of the cutoff level
as any full subtree $D_\tau'$ used to build a disjoint decomposition must have depth greater or equal than $k$, the depth of the cutoff
level.
\end{proof}

\newcommand{\quat}{QAut}($\mathcal{T}_{2,c}$)
\section{An embedding $V \hookrightarrow \mathrm{QAut}(\mathcal{T}_{2,c})$}
In this section we provide an embedding $\Theta:  V\hookrightarrow \mathrm{QAut}(\mathcal{T}_{2,c})$.  Our embedding is similar in spirit to the one described by Lehnert in his dissertation \cite{LehnertDissertation}.  We provide this embedding as we were not able to directly verify the embedding Lehnert describes.

Given an element of $V$, we will embed it in $\mathrm{QAut}(\mathcal{T}_{2,c})$, with support over the union of the set of words which begin with `$0$' together with the set which contains only the  empty word.  No word beginning with `$1$' will be moved by our embedding.
Intuitively, our map will be what one gets if one associates $V$ as acting on the ordered set $(0,1/2]$ by interval exchange maps which exchange intervals of the form $(a,b]$ where $a$ and $b$ are dyadic rationals in the set $(0,1/2]$ (where here, the root of $\mathcal{T}_{2,c}$ is corresponding to the value $1/2,$ and our embedded copy of $V$ is acting only on the left half of the interval $[0,1]$ (fixing $0$).  Thus, our embedding really will be in the spirit of Lehnert's embedding, we now formalise this discussion.

Rule for the injection:  Given any finite rooted binary tree $T$, overlay the tree on $\mathcal{T}_{2,c}$ so that the root of $T$ will be placed at the node $0$ of $\mathcal{T}_{2,c}$.  Refer to the embedded tree as $T'$.  Let $X'$ be the set of interior nodes of $T'$ as a subset of the nodes of $\mathcal{T}_{2,c}$, and set $X:=X'\cup\{\varepsilon\}$, that is, $X'$ together with the empty node.  Now associate a bijection $\omega_T$ from the leaves of $T'$ to the nodes in $X$.  We associate the leaves of $T'$ to the nodes of $X$ in left-to-right order (as seen in the tree where again, $0$ means ``left child" and $1$ means ``right child").  In particular, the rightmost leaf of $T'$ is associated with the node $\epsilon.$

Now, given an element $\alpha\in V$, let us describe the image $\alpha\Theta$ in $\mathrm{QAut}(\mathcal{T}_{2,c}).$  Suppose $\alpha\sim(D,R,\sigma)$.  The we can embed both $D$ and $R$ in $\mathcal{T}_{2,c}$ each as in the above paragraph to find embedded images $D'$ and $R'$, remembering the associations from the leaves of these trees to the finite words in $\{0,1\}^*$ which correspond to nodes of these trees (or the empty word) after they are embedded in $\mathcal{T}_{2,c}$.  Now the permutation $\sigma$ informs us how to move the maximal subtrees of the tree $\mathcal{T}_{2,c}$ rooted at the leaves of $D'$ so they become trees rooted at the leaves of $R'$.  If $n$ is a node of $\mathcal{T}_{2,c}$ associated to a leaf of $D'$ via $\omega_D^{-1}$ (so, $n$ is either an interior node of $D'$ or the root node of $\mathcal{T}_{2,c}$), then it should be mapped to the node $n\omega_D^{-1}\sigma\omega_R$ which is a node of $\mathcal{T}_{2,c}$ associated to a leaf of $R'$ by the map $\omega_R$.

The reader may now verify that the proposed construction is well defined, and results in an injective group homomorphism from $V$ into $\quatT$.  In checking well-defined-ness, the authors proved the following lemma.

\begin{lemma}\label{wd_VinQ}
Suppose that $\alpha\in V$ has $\alpha\sim(D_1,R_1,\sigma_1)$ and $\alpha\sim(D_2,R_2,\sigma_2)$, where $D_2$ is an elementary expansion of $D_1$ (that is, $D_2$ contains $D_1$, and has exactly one extra caret).  Then $R_2$ is an elementary expansion of $R_1$ as well, and the map $\Theta$ will send both representative tree pairs $(D_1,R_1,\sigma_1)$ and $\alpha\sim(D_2,R_2,\sigma_2)$ to the same element in \quat. \end{lemma}

The authors have modelled both $\quatT$ and $V$ in \texttt{GAP}, and using this verified that $\Theta$ extends to a group homomorphism which preserves the relations in $V$.


\section{An embedding $\mathrm{QAut}(\mathcal{T}_{2,c}) \hookrightarrow V$}
\label{sec:QAut-inside-V}

In this section we show the existence of a group homomorphism $\varphi: \quatT \to V$ and,
for $\tau \in \quatT$, we define it via the the disjoint decomposition $d_\tau$.

For every $\tau \in \mathrm{QAut}(\mathcal{T}_{2,c})$, we apply Lemma \ref{thm:disjoint-decomposition} and obtain a 
disjoint decomposition of $\tau$ as an element $d_\tau \in \mathrm{QAut}(\mathcal{T}_{2,c})$.
We let $D_{d_\tau}$ and $R_{d_\tau}$ be the domain and range trees so that $d_\tau$ acts as an element of $V$ on the leaves of $D_{d_\tau}$
(and below them) and as a bijection above such leaves.

We now apply the following construction to the trees $D_{d_\tau}$ and $R_{d_\tau}$. We only explain it for the tree $D_{d_\tau}$, the other being analogous.
Assume that a vertex $w\in D_{d_\tau}$ has two children edges $e_{\mathrm{left}}$ and $e_{\mathrm{right}}$ and one parent edge 
$e_{\mathrm{parent}}$ in the tree.
We replace $w$ with a caret whose vertices are labeled by $(w, w_n, w_p)$, where $w_n$ is the left child of $w$ and $w_p$ is the right child of
$w$.
We attach the former parent edge $e_{\mathrm{parent}}$ on top of $w$ and we attach the former edges $e_{\mathrm{left}}$ and $e_{\mathrm{right}}$
below the vertex $w_n$. The vertex $w_p$ has no children. We apply this construction to every vertex 
$w \in D_{d_\tau}$ with two exceptions: 
\begin{enumerate}
\item the root vertex $\varepsilon$, for which there is no parent edge and so
we only attach the two edges below the left child $\varepsilon_n$ of $\varepsilon$.
\item Any leaf vertex $w$, to which we attach the two children edges with terminal vertices $w_n$ and $w_p$.
\end{enumerate}
The result of this construction is shown in figure \ref{fig:caret}.
We denote the two trees we have just constructed by
$\widehat{D}_{d_\tau}$ and $\widehat{R}_{d_\tau}$.
We say that all leaves of the form $w_n$ are the \textbf{$n$-leaves} and the leaves of the form $w_p$
are the \textbf{$p$-leaves}.

\begin{center}
\begin{figure}\label{fig:caret}
\includegraphics[height=2cm]{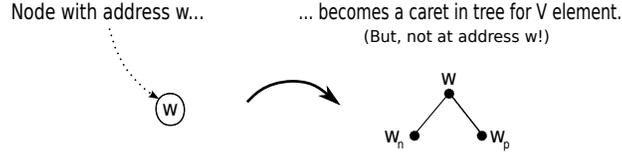}
\caption{The replacement rule for nodes}
\end{figure}
\end{center}

If $\LF(T)$ defines the set of leaves of a finite binary tree $T$, we observe that the map $d_\tau$ induces a a bijection $\sigma_{d_\tau}: \LF(\widehat{D}_{d_\tau}) \to \LF(\widehat{R}_{d_\tau})$ which we now
construct. By definition, 
the map $\sigma_{d_\tau}$ acts on a leaf $w_n \in \LF(\widehat{D}_{d_\tau})$ 
by sending it to the leaf $d_\tau(w)_n \in \LF(\widehat{R}_{d_\tau})$
and 
acts on a leaf $w_p \in \LF(\widehat{D}_{d_\tau})$ by sending it to a suitable leaf $t_p \in \LF(\widehat{R}_{d_\tau})$ if the map $d_\tau$ sends its parent vertex $w$ (seen as a vertex in $D_{d_\tau}$) to a vertex $t$ (seen as a vertex in $R_{d_\tau}$).
This defines a permutation $\sigma_{d_\tau}$ on the $n$-leaves (which comes from the permutation of the $V$-part of $d_\tau$) and on the $p$-leaves
(which comes from the bijection part of $d_\tau$).

We define $\varphi(\tau):=(\widehat{D}_{d_\tau}, \widehat{R}_{d_\tau}, \sigma_{d_\tau}) \in V$
(see figure \ref{fig:QAutInV_both}). 
We need to verify that such map is well defined as it relies on a choice of a disjoint
decomposition.

\begin{center}
\begin{figure}\label{fig:QAutInV_both}
\includegraphics[height=10cm]{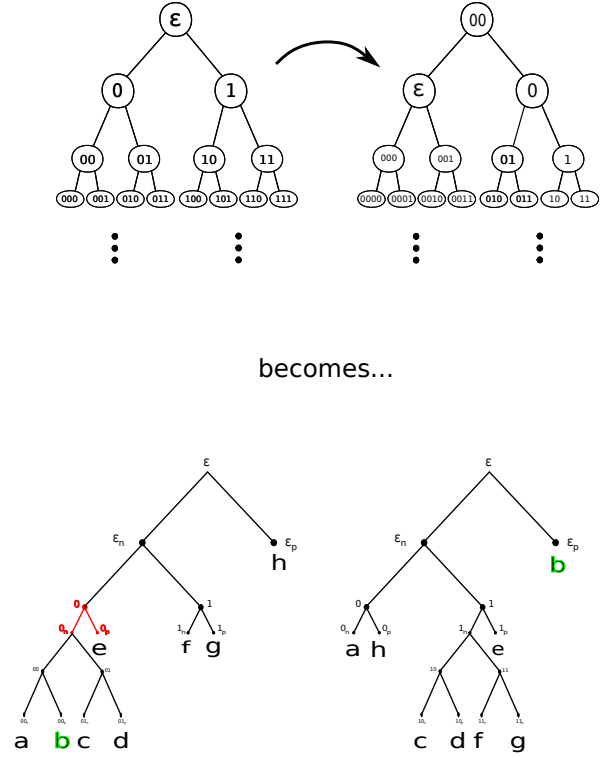}
\caption{The replacement rule for nodes}
\end{figure}
\end{center}

\begin{lemma}\label{thm:well-defined}
The map $\varphi: \quatT \to V$ defined above is well-defined.
\end{lemma}

\begin{proof}
By Lemma \ref{thm:cutoff-disjoint}, any disjoint decomposition $d_\tau$ is built by starting from the cutoff disjoint 
decomposition and expanding the domain full subtree of the $V$-part coincides with $D_{d_\tau}$.
Any disjoint decomposition represents exactly the same map as $\tau$, but the way it is written out is slightly different, 
even if the final outcome is the same function. Any two disjoint decomposition can be obtained from the cutoff one using unreductions
of the $V$-part, thus one can go from one disjoint decomposition to another via a sequence of reductions and unreductions of the $V$-part.
Hence, it is sufficient to prove our claim in the case of
two disjoint decompositions $f_\tau$ and $g_\tau$ so that the tree pair defining the $V$-part of  $g_\tau$ is obtained by
an unreduction of a single level, that is adding a caret to each leaf in the domain and the range
tree of the $V$-part of $f_\tau$. 

Let $w$ be one of the leaves of $f_\tau$ to which we added two new children $w0$ and $w1$.
By construction, the domain (respectively, range) tree of 
$\varphi(g_\tau)$ is obtained by adding the construction of figure \ref{fig:caret}
to the leaf $w_n$ (respectively, adding the same construction to $f_\tau(w)_n=g_\tau(w)_n$). Moreover, the vertices $w_p, w0_p, w1_p$ are mapped
in an order preserving way 
to the vertices $f_\tau(w)_p=g_\tau(w)_p, g_\tau(w0)_p, g_\tau(w1)_p$.
Observe that, as we added the exact same construction of figure \ref{fig:caret} to the vertices $w_n$ and $f_\tau(w)_n$ while building the tree diagram of $\varphi(g_\tau)$, we can immediately
reduce it. We can do this for each and every of the carets we added to the $V$-part of $f_\tau$. Hence, it is immediate that $\varphi(g_\tau)=\varphi(f_\tau)$.
\end{proof}

\begin{theorem}
The map $\varphi: \quatT \to V$ is an injective homomorphism.
\end{theorem}

\begin{proof}
\emph{The map $\varphi$ is a group homomorphism}. Let $\tau, \lambda \in \quatT$ and let
$f_\tau$ and $g_\lambda$ be disjoint decompositions built so that the range tree $R_{f_\tau}$ contains the full subtree constituting the domain tree $D_{g_\lambda}$.
We can now unreduce the domain tree $D_{g_\lambda}$ to make it become equal to $R_{f_\tau}$ and we unreduce $R_{g_\lambda}$ accordingly to get a new tree $R$. We call $\overline{\sigma}_{g_\lambda}$ the bijection
on the vertices which one obtains from $\sigma_{g_\lambda}$ after this unreduction.
We define $g_\lambda' \in \quatT$ as the map which has a $V$-part defined 
by $(R_{f_\tau}, R, \overline{\sigma}_{g_\lambda})$ and a bijection part which is given by $g_\lambda$ on all vertices above the set
$\LF(R_{f_\tau})$. Therefore,
$g_\lambda' = g_\lambda$ as maps and, by a slight abuse of notation, we still say that $g_\lambda'$ is a ``disjoint''
decomposition of $\lambda$ as in this proof we only need the requirement that
$g_\lambda'$ is an automorphism on every vertex below $\LF(D_{g_\lambda'})$.
Therefore we have
\[
(D_{f_\tau}, R_{f_\tau}, \sigma_{f_\tau}) (D_{g_\lambda}, R_{g_\lambda}, \sigma_{g_\lambda})=
(D_{f_\tau}, R_{f_\tau}, \sigma_{f_\tau}) (R_{f_\tau}, R, \overline{\sigma}_{g_\lambda})=
(D_{f_\tau}, R, \sigma_{f_\tau} \overline{\sigma}_{g_\lambda})
\]
By construction, the tree pair $(D_{f_\tau}, R, \sigma_{f_\tau} \overline{\sigma}_{g_\lambda})$ 
has a domain tree which is a full subtree and is deep enough so that the action on all levels below is given by the composition
of the $V$-parts of $f_\tau$ and of $g_\lambda$. 
This implies that the tree pair $(D_{f_\tau}, R, \sigma_{f_\tau} \overline{\sigma}_{g_\lambda})$
constitutes the $V$-part of a disjoint decomposition for the element $\tau\lambda$ (which we denote by $d_{\tau \lambda}$), where the associated bijection part 
can be computed by composing the bijections occurring within the trees $D_{f_\tau}, R_{f_\tau}, R$.

Now observe that $\varphi(d_{\tau \lambda})$ is obtained by expanding the trees $D_{f_\tau}$ and $R$ and applying the $d_{\tau \lambda}$ to the $n$-leaves 
and the $p$-leaves.
We obtain the tree pair $(\widehat{D}_{f_\tau}, \widehat{R}, \sigma)=\varphi(d_{\tau \lambda})$.

We consider now the tree pairs $\varphi(f_{\tau})$ and $\varphi(g_{\lambda}')=\varphi(g_{\lambda})$. 
Since the range tree of the $V$-part of $g_\lambda'$ is $R$, it is clear that the range tree of 
$\varphi(g_{\lambda}')$ is equal to $\widehat{R}$.
Therefore, the domain tree of $\varphi(d_{\tau \lambda})$ and of 
$\varphi(f_{\tau})\varphi(g_{\lambda}')$ is equal and the corresponding range trees coincide.


We observe that the bijection on the $n$-leaves of $\varphi(d_{\tau \lambda})$ and 
$\varphi(f_{\tau})\varphi(g_{\lambda}')$ is exactly the same because the $V$-part of $d_{\tau \lambda}$ was built
by composing the tree pairs of the $V$-parts of $f_{\tau}$ and $g_{\lambda}'$. 
As we are dealing with disjoint decompositions, the actions above $\LF(D_{f_\tau})$ and $\LF(D_{g_\lambda'})$ do not affect the $V$-parts and so the bijections on the $p$-leaves of $\varphi(f_{\tau})$ and of $\varphi(g_{\lambda'})$ 
are exactly the bijections appearing in the bijection parts of $f_\tau$ and $g_\lambda'$. Thus the bijection on the $p$-leaves of 
$\varphi(f_{\tau})\varphi(g_{\lambda}')$ is determined by the composition of the 
bijection parts in $f_\tau$ and $g_\lambda'$ which is equal to the bijection part of $d_{\tau \lambda}$. Therefore the bijection on the $p$-leaves of 
$\varphi(f_{\tau})\varphi(g_{\lambda}')$ is equal to that of $\varphi(d_{\tau \lambda})$.
By putting everything together and using Lemma \ref{thm:well-defined} we deduce that $\varphi(d_{\tau \lambda}) =
\varphi(f_{\tau})\varphi(g_{\lambda}') = \varphi(f_{\tau})\varphi(g_{\lambda})$.

\bigskip
\noindent \emph{The map $\varphi$ is injective.} 
Let $d_\tau$ be the cutoff decomposition form for $\tau \in \ker \varphi$. 
The domain and the range tree of $\varphi(d_\tau)$ must be the equal and the permutation on every leaf is the identity permutation. Therefore 
the $V$-part of $d_\tau$ is the identity element and the bijection part is the identity map, that is $d_\tau$
is the identity map on $\cT$ and so $\ker \varphi = \{1_{\quatT}\}$.
\end{proof}

\section{Embedding Baumslag-Solitar groups in $V$\label{BSEmbeddings}}

In this section, we prove that the Baumslag-Solitar groups 
$$BS(m,n)=\langle a,b\mid  b^{-1}a^m b=a^n\rangle$$ fail to embed in R. Thompson's group $V$ whenever $|m|\neq |n|$ and that they do embed otherwise. 
The following result was observed by R\"over as a consequence of a result by Higman
\cite{hig}.

\begin{theorem}[R\"over, \cite{roverthesis}]
If $n$ is a proper divisor of $m$, then the group $BS(m,n)$ does not embed in Thompson's group $V$.
\end{theorem}

Recall that Farb and Franks 
have embedding results in groups related to Thompson groups.  More precisely,
they show the following 
\begin{theorem}[Farb-Franks, \cite{FarbFranksI}]
Let $m,n$ be positive integers.
\begin{enumerate}
\item If $m>n\ge1$ the group $BS(m,n)$ embeds in the group of 
orientation-preserving analytic
diffemorphisms $\mathrm{Diff}_+^\omega(\mathbb{R})$ and
also inside the
groups of orientation-preserving homeomorphisms $\mathrm{Homeo}_+(S^1)$
and $\mathrm{Homeo}_+(I)$. 
\item If $m> n>1$, the group $BS(m,n)$ does not embed into $\mathrm{Diff}_+^2(I)$.
If $n$ does not divide $m$, the group $BS(m,n)$ does not
embed into $\mathrm{Homeo}_+(S^1)$.
\end{enumerate}
\end{theorem}
Taking the Farb and Franks result together with the R\"over result, we have some evidence that most Baumslag-Solitar groups do not embed in $T$, and probably also in $V$.  Thus, it is somewhat natural that only if $|m|= |n|$ should one expect that $BS(m,n)$ might embed  in $V$ (well, in R. Thompson's group $T$ at the least).  In this section, we show that these indications do not mislead.

 The key idea behind the main result of this section is the following:

\begin{lemma}\label{conjugationAndPowers}
Let $v\in V$ be non-torsion, and $r$ and $s$ be integers.  Then whenever $(v^r)^w=v^s$ for some $w\in V$, we have that $|r|=|s|$.
\end{lemma}

The essence of the argument below will be clear to any reader who has digested the material on revealing pairs for elements of $V$ (for instance, as presented in 
\cite{matucci8}, which has an expository section written to explain these objects).  Note that revealing pairs are introduced by Brin in \cite{brin4}, and that Higman in \cite{hig} had already developed an analogous technology.

\begin{proof}
If two elements $x,y \in V$, we use the conjugation notation $x^y=y^{-1}xy$.
Assume that $v$, $w\in V$, $v$ is not torsion, and that there are integers $r$ and $s$ so that $(v^r)^w=v^s.$  We will now show that $|r|=|s|$. 
Note that if $r=0$ the lemma result is immediately true, as $v^0=1_V$.  Hence, we will assume below that neither $r$ nor $s$ is zero.  

As $v$ is not torsion, by an extension of an argument of Brin in \cite{brin4} there is a minimal positive integer $m$ so that, if we set $\alpha:=v^m$, then $\langle \alpha\rangle$ is an infinite cyclic group embedded in $V$ with the property that this group acts with no non-trivial finite orbits on the Cantor set $\CS_2$. For example, a possibility for $m$ is the least common multiple of the set of lengths of all finite periodic orbits, as there are only finitely many such lengths. 
Now, the element $\alpha$ admits a finite set of points $I(\alpha)$, which we will call the \emph{important points of $\alpha$} (following \cite{matucci8}), consisting of the repelling and attracting points in the Cantor set under the action of $\langle \alpha\rangle$.  For each point in the set $I(\alpha)$, it is the case that $\alpha$ restricted to some small 
interval $U_p$ containing $p$ 
is an affine map which fixes exactly the point $p$, where the slope of this map is $2^{s_p}$, for some  $s_p$ a fixed non-zero integer.
Now, as $(v^{r})^w=v^s$ we see that $$(\alpha^r)^w=((v^m)^r)^w=((v^r)^w)^m= (v^s)^m=(v^m)^s=\alpha^s$$ as well. 

For any integer $u \in \mathbb{Z} \setminus \{0\}$, 
consider the finite set of logarithms of derivatives
\[
\mathcal{S}_u = 
\left\{ \log_2 \frac{d}{dt} \alpha^u \Big\vert_{t=x} \; \; \mid \; \; x \in I(\alpha) \right\}.
\]
Observe that $\mathcal{S}_u=u \cdot \mathcal{S}_1
=\{u v \mid v \in \mathcal{S}_1 \}$.
The equation above implies that $w$ sends $\mathcal{S}_r$ to $\mathcal{S}_s$. 
It is straightforward to verify that the finite set of slopes of $(\alpha^r)^w$ on $I(\alpha)$
is exactly equal to $\mathcal{S}_r$.
Therefore
\[
r \cdot \mathcal{S}_1 = \mathcal{S}_r = \mathcal{S}_s = s \cdot \mathcal{S}_1
\]
and thus, if $k = \max\{|v| \mid v \in \mathcal{S}_1 \}>0$, one has
\[
|r| k = \max \{|r v| \mid v \in \mathcal{S}_1\}=
\max \{|s v| \mid v \in \mathcal{S}_1\} = |s| k
\]
By the cancellation law, we have $|r|=|s|$.
\end{proof}

\begin{remark}
We would like to make a small historical comment.  The core idea behind the proof 
of Lemma \ref{conjugationAndPowers} is that the product of the set of slopes of the affine restrictions of $v$ in a small neighbourhood of an orbit of a repelling or attracting periodic point (of $\CS_2$  associated with  the action of $\langle v\rangle$) is an invariant of conjugacy in $V$ of $v$. (See, e.g., \cite{matucci9, gusa1, hig, salazar1} for solutions of the conjugacy problem for $V$).  One can think of the result as a total ``speed'' along an orbit, and it is useful in many ways. For instance, this product-of-slopes calculation for an orbit is used in section 7 of
\cite{matucci8} to analyse element centralisers.  As described later in this section, it is also used in \cite{matucci8} to show that all cyclic subgroups are undistorted in $V$.
\end{remark}

We are now ready to give a proof of Theorem \ref{BS-groups}.

\medskip
\noindent
\textbf{Theorem \ref{BS-groups}.}
\emph{
Let $m,n \in \mathbb{Z} \setminus \{ 0 \}$. Let $BS(m,n)$ be the corresponding
Baumslag-Solitar group.
\begin{enumerate}
\item If $|m| \ne |n|$, then $BS(m,n)$ fails to embed in $V$.
\item If $|m| = |n|$, then there is an embedding of $BS(m,n)$ in $V$.
\end{enumerate}
}

\begin{proof}
Part (1) follows immediately by Lemma \ref{conjugationAndPowers}, since
the element $a \in BS(m,n)$ is non-torsion.

For part (2) we rely on the following observation which we learned from Yves de Cornulier
on MathOverflow \cite{cornulier-mathoverflow}.
Let $e \in \{-1,1\}$ and consider the group $BS(m,em)$. We now 
consider the diagonal embedding of $BS(m,em)$ 
into the direct product $\left(\mathbb{Z} \ast (\mathbb{Z}/m\mathbb{Z})\right)\times 
\left(\mathbb{Z}\rtimes_{\varphi}\mathbb{Z} \right)$.
The embedding of $BS(m,em)$ in the two coordinates is seen as follows:
\begin{itemize}
\item 
the left homomorphism is given by modding out by the common subgroup 
$\langle a^m \rangle$ inside $BS(m,em)$,
\item the right homomorphism is the homomorphism
from $BS(m,em)=\langle a,b \rangle$ to $BS(1,e)=\langle a',b' \rangle
\cong \left(\mathbb{Z}\rtimes_{\varphi}\mathbb{Z} \right)$, where 
$a \mapsto a'$ and $b \mapsto b'$.
\end{itemize}
It is straightforward to see that the intersection of the kernels of the two coordinates
is trivial.
Now, the groups $\mathbb{Z}$ and $\mathbb{Z}/m\mathbb{Z}$ are demonstrative
subgroups of $V$
(see Bleak and Salazar \cite{bleak-salazar1} for the definition of \emph{demonstrative
subgroups}), so Theorem 1.4 in \cite{bleak-salazar1} implies
that $\mathbb{Z} \ast (\mathbb{Z}/m\mathbb{Z})$ is a subgroup of $V$.

Separately, we need to discuss the following two cases:
\begin{itemize}
\item the group $BS(1,1)$, which is isomorphic to $\mathbb{Z}^2$ is a well-known subgroup of $V$, and
\item the group $BS(1,-1)$ which is the Klein bottle group.

Noting that there is a double cover of the torus to the Klein bottle, 
the Klein bottle group is a finite extension of the group
$\mathbb{Z}^2 \le V$ and therefore $BS(1,-1)$ is also
a subgroup of $V$.
\end{itemize}
Therefore, $BS(m,em)$ embeds into the group 
$\left(\mathbb{Z} \ast (\mathbb{Z}/m\mathbb{Z})\right)\times 
\left(\mathbb{Z}\rtimes_{\varphi}\mathbb{Z} \right)$ which is a subgroup of $V \times V$
which, in turn, is a subgroup of $V$.
\end{proof}

%

We are grateful to Jos\'e Burillo who pointed out that one can sometimes use an argument based on distortion to show certain non-embedding results.  Indeed, this type of argument can be used in our context to show the main non-embedding results of Theorem \ref{BS-groups}.  Here is how the argument runs.

First, recall Theorem 1.3 in \cite{matucci8} which says that all cyclic groups are undistorted in $V$.  Now recall that when $|m|\neq |n|$, the group $BS(m,n)$ has distorted cyclic subgroups.  Now we observe that if $BS(m,n)$ were a subgroup of $V$ and
$v \in BS(m,n)$ were an element of infinite order, then the distortion of $\langle v \rangle$
in $V$ would have to be at least as much as it is in $BS(m,n)$ and that yields a contradiction,
therefore implying that for such $m$ and $n$, $BS(m,n)$ does not embed in $V$.

We observe that this argument is fundamentally  equivalent to the one given in the proof of Lemma \ref{conjugationAndPowers};  to prove cyclic groups are undistorted in $V$, one measures the `speed' of the elements in the cyclic subgroup, near to their attracting and repelling orbits.

\bibliographystyle{amsplain}
\bibliography{go}

\end{document}